\documentclass[12pt, a4paper]{article}
\usepackage[left=2.8cm, right=2.8cm]{geometry}
\usepackage[utf8]{inputenc}
\usepackage{amsmath, amsthm, amssymb}
\usepackage{thmtools}
\usepackage{xfrac}
\usepackage{bm}
\usepackage{xcolor}
\usepackage{url}
\usepackage{footmisc}

\numberwithin{equation}{section}

\newtheorem{theorem}[equation]{Theorem}
\newtheorem{lem}[equation]{Lemma} 

\newtheorem{cor}[equation]{Corollary}

\newtheorem{conj}[equation]{Conjecture}

\theoremstyle{definition}
\newtheorem{example}[equation]{Example}

\newtheoremstyle{case}{}{}{\normalfont}{}{\itshape}{\normalfont:}{ }{}

\theoremstyle{case}
\newtheorem{case}{Case}

\newtheorem{remark}[equation]{Remark}
\usepackage{mathtools}

\newcommand{\floor}[1]{\left\lfloor #1 \right\rfloor}
\newcommand{\ceil}[1]{\left\lceil #1 \right\rceil}

\newcommand{\F}{\mathcal{F}}
\newcommand{\G}{\mathcal{G}}
\newcommand{\A}{\mathcal{A}}
\newcommand{\X}{\mathcal{X}}
\newcommand{\Y}{\mathcal{Y}}

\newcommand{\s}{\sigma}

\newcommand{\maxprod}{\pi}

\title{Improved bounds for cross-Sperner systems}
\author{
    Natalie Behague\footnotemark[1]\thanks{Department of Mathematics and Statistics, University of Victoria. \newline E-mail: \texttt{\{nbehague,akuperus,nmorrison,ashnawright\} @uvic.ca}.}  \thanks{Research supported by a PIMS Postdoctoral Fellowship} \and
    Akina Kuperus\footnotemark[1] \thanks{Research supported by NSERC CGS-M.}
    \and Natasha Morrison\footnotemark[1] \thanks{Research supported by NSERC Discovery Grant RGPIN-2021-02511 and NSERC Early Career Supplement DGECR-2021-00047}
	\and Ashna Wright\footnotemark[1] \thanks{Research supported by NSERC USRA.}
	}

\begin{document}

\maketitle
\begin{abstract}
    A collection of families $(\F_{1}, \F_{2} , \cdots , \F_{k}) \in \mathcal{P}([n])^k$ is \emph{cross-Sperner} if there is no pair $i \not= j$ for which some $F_i \in \F_i$ is comparable to some $F_j \in \F_j$. 
    Two natural measures of the `size' of such a family are the sum $\sum_{i = 1}^k |\F_i|$ and the product $\prod_{i = 1}^k |\F_i|$. We prove new upper and lower bounds on both of these measures for general $n$ and $k \ge 2$ which improve considerably on the previous best bounds. In particular, we construct a rich family of counterexamples to a conjecture of  Gerbner, Lemons, Palmer, Patk\'{o}s, and Sz\'{e}csi from 2011.
    
\end{abstract}

\section{Introduction}
A family $\F \subseteq \mathcal{P}([n])$ is an \emph{antichain} (also known as a \emph{Sperner} family) if for all distinct $ F, G \in \F$, neither $F \subseteq G$ nor $G \subseteq F$ (i.e.~$F$ and $G$ are \emph{incomparable}). 
One of the principal results in extremal combinatorics is Sperner's theorem \cite{Sperner}, which states that the largest size of an antichain in $\mathcal{P}([n])$ is $\binom{n}{\floor{n/2}}$. This can be seen to be tight by taking a `middle layer', that is $\F = \binom{[n]}{\floor{n/2}}$ or $\F = \binom{[n]}{\lceil n/2 \rceil}$.

It is natural to consider a generalisation of Sperner's theorem to multiple families of sets. For $k \ge 2$, say that a collection of non-empty 
families $(\F_{1}, \F_{2} , \cdots , \F_{k}) \in \mathcal{P}([n])^k$, is \emph{cross-Sperner} if for all $i \not= j$, the sets $F_i$ and $F_j$ are incomparable for any $F_i \in \F_i$ and $F_j \in \F_j$. 
(We may also write that $(\F_{1}, \F_{2} , \cdots , \F_{k})$ is \emph{cross-Sperner in} $\mathcal{P}([n])$.) 
The study of such objects goes back to the 1970s when Seymour~\cite{seymour} deduced from a result of Kleitman~\cite{kleitman} that a cross-Sperner pair $(\F, \G)$ in $\mathcal{P}([n])$ satisfies 

\begin{equation}\label{eq:sumsqrt}
    |\F|^{1/2} + |\G|^{1/2} \leq 2^{n/2},
\end{equation} hence resolving a related conjecture of Hilton (see~\cite{brace}). Equality is obtained in Seymour's bound precisely when the minimal sets of $\F$ are pairwise disjoint from the minimal sets intersecting each set of $\G$. A broad spectrum of research concerning discrete objects with `Sperner-like' properties have since emerged (see, for example, \cite{randomsperners4, randomsperner, randomsperner2, eng, maximalantichain3,  maximalantichain2, maximalantichains, maximalantichains4, maximalantichains5, randomsperner3, west}). Many related results concern families satisfying both Sperner-type properties, and additional properties such as conditions on intersections (see, for example~\cite{FUREDI, Kleitmancrossinter, lz, MATSUMOTO,  PYBER}). 

Let $\F = (\F_{1}, \F_{2} , \cdots , \F_{k})$ be cross-Sperner in $\mathcal{P}([n])$. There are several natural measures of the `size' of such a family. These include the sum $\sum_{i = 1}^k |\F_i|$ and the product $\prod_{i = 1}^k |\F_i|$. The general study of these quantities was initiated by Gerbner, Lemons, Palmer, Patk\'{o}s, and Sz\'{e}csi~\cite{xsperner}, who essentially proved best possible bounds on cross-Sperner \emph{pairs} of families. 

Concerning the product, they gave a direct proof that a cross-Sperner pair $(\F, \G)$ in $\mathcal{P}([n])$ satisfies
\begin{equation}\label{eq:prod}
  |\F|\cdot |\G| \le 2^{2n - 4}. 
\end{equation}
 To see that this bound is tight, consider $\F = \{ F \subseteq [n] : 1 \in F, n \not\in F\}$ and $\G = \{G \subseteq [n] : 1 \not\in G, n \in G\}$. It is straightforward to see that the bound in \eqref{eq:prod} can also be obtained as a direct consequence of \eqref{eq:sumsqrt} via the AM--GM inequality\footnote{Observe that  $2\left(|\F||\G|\right)^{1/4}  \le |\F|^{1/2} + |\G|^{1/2} \le 2^{n/2}$. }. 

First, let us focus on product bounds for $k \ge 3$. It is convenient to define
\[
    \maxprod(n,k):= 
    \max\left\{
        \prod_{i = 1}^k |\F_i|: (\F_{1}, \F_{2} , \cdots , \F_{k}) \text{ is cross-Sperner in } \mathcal{P}([n])
    \right\}.
\]
In~\cite{xsperner}, it was observed that \eqref{eq:sumsqrt} can be used to obtain the upper bound $\maxprod(n,k) \leq 2^{k(n -2)}$. For $k > 4$, an improved bound of $\maxprod(n,k) \le \left(\frac{2^n}{k} \right)^k$ can be obtained by a simple application of the AM-GM inequality.\footnote{Similarly to above, we have $\left(\prod_{i = 1}^k |\F_i|\right)^{\frac{1}{k}} \le \frac{\sum_{i=1}^k |\F_i|}{k} \le \frac{2^n}{k}$.} 
Gerbner, Lemons Palmer, Patk\'{o}s, and Sz\'{e}csi~\cite{xsperner} conjectured that $\maxprod(n,k) \leq 2^{k(n-\ell^*)}$, where $\ell^* = \ell^*(k)$ is the least positive integer such that $\binom{\ell^*}{\floor{\ell^*/2}} \geq k$. 
They described a construction which provides a matching lower bound to their conjecture. Let $A_1,\ldots, A_k$ be an antichain in $\mathcal{P}([l])$ and let $(\F_1,\ldots, \F_k) \in \mathcal{P}([n])$ be defined by $\F_{i} := \{F \in [n] : F \cap [\ell] = A_i\}$. 

Our first theorem strongly disproves this conjecture. 

\begin{restatable}{theorem}{prodl}
\label{thm:prodl}
Let $n$ and $k \ge 2$ be integers. For $n$ sufficiently large, 
\[
    \left(\frac{2^n}{ek}\right)^k \le \maxprod(n,k).
\]
\end{restatable}

A crude application of Stirling's approximation yields that $\ell^*(k) = \omega(\log k)$. So in particular, there is a function $g(k)$ tending to infinity with $k$ such that  $2^{k(n-\ell^*)} = O\left( 2^{kn}(k \cdot g(k))^{-k}\right)$. 
Therefore our lower bound is exponentially larger than the conjectured $2^{k(n- \ell^*)}$.

We also improve the previous best known upper bound by a factor of $2^k$.
\begin{restatable}{theorem}{produ}
\label{thm:produ}
Let $n$ and $k \ge 2$ be integers. Then 
\[
    \maxprod(n,k) \le \left( \frac{2^n}{k^2}\right)^{k}\floor{\frac{k}{2}}^{\floor{k/2}}\ceil{\frac{k}{2}}^{\ceil{k/2}}
\]
\end{restatable}

Regarding bounds on the sum, in~\cite{xsperner} it is shown that for $n$ sufficiently large, a cross-Sperner pair in $\mathcal{P}([n])$ satisfies 
\begin{equation}\label{eq:sumpair}
    |\F| + |\G| \le 2^{n} - 2^{\ceil{n/2}} - 2^{\floor{n/2}} + 2.
\end{equation} 
This is tight, which can be seen by taking $\F = \{1,2, \ldots, \floor{n/2}\}$ and letting $\G$ be all subsets of $[n]$ that are not comparable to $F$. Gerbner, Lemons Palmer, Patk\'{o}s, and Sz\'{e}csi~\cite{xsperner} also asked about bounds for the sum for general $k$. Analogously to in the product case, define
\[
\sigma(n,k):= \max\left\{\sum_{i = 1}^k |\F_i|: (\F_{1}, \F_{2} , \cdots , \F_{k}) \text{ is cross-Sperner in } \mathcal{P}([n])\right\}.
\]
In our next theorem, we determine upper and lower bounds on $\sigma(n,k)$. 
\begin{theorem}\label{thm:cleansum}
Let $n,k$ be integers with $n \ge 2k$. Then
\[
2^{n}  - \frac{3}{\sqrt{2}}\sqrt{2^n k} + 2(k-1)
~\le~  \sigma(n,k)
~\le~ 2^{n} - 2\sqrt{2^n (k-1)} + 2(k-1).
\]
\end{theorem}

When $k$ is a power of 2 and $n - \log_2{k}$ is even, we can further improve the lower bound to
$2^{n}  - 2\sqrt{2^n k} + 2(k-1)$, which is extremely close to the upper bound.

In order to prove Theorem~\ref{thm:produ} and the upper bound in Theorem~\ref{thm:cleansum}, we exploit a connection between $\sigma(n,k)$ and the \emph{comparability number} of a set (given in Section~\ref{sec:comp}). In doing so, we recover a simple proof of \eqref{eq:sumpair} (see Theorem~\ref{thm:sumpair}) that holds for \emph{all} $n$ (recall the result of \cite{xsperner} holds for large $n$). 

The article is structured as follows. We introduce the comparability number in Section~\ref{sec:comp} and provide a lower bound (Theorem~\ref{thm:compnum}) that will be used in the proofs of Theorems~\ref{thm:produ} and \ref{thm:cleansum}. In Section~\ref{sec:prod} we prove Theorems~\ref{thm:prodl} and \ref{thm:produ} bounding the product. In Section~\ref{sec:sum} we prove Theorem~\ref{thm:cleansum} bounding the sum. We conclude in section~\ref{sec:conc} with some discussion and open questions. 




\section{Minimizing Comparability} \label{sec:comp}
Given a family $\F \subseteq \mathcal{P}([n])$ define the \emph{comparability number of} $\F$ to be
\[
    c(n,\F) := |\{X \subseteq [n] : X \text{ is comparable to some } A \in \F \}|.
\]
When the setting is clear from context, we may write $c(\F)$ for $c(n,\F)$.
Define 
\[
    c(n,m) = \min\{c(n,\F) : \F \subseteq \mathcal{P}([n]), |\F| = m\}.
\]

As noted in \cite{xsperner}, there is a direct relationship between $\sigma(n,2)$ and $c(n,m)$. Observe that if $(\F,\G)$ is cross-Sperner in $\mathcal{P}([n])$, we have
$$|\F| + |\G| \le |\F| + 2^n - c(n,|\F|),$$
as any set incomparable to every member of $\F$ can be added to $\G$. We will use analogous ideas in Section~\ref{sec:sum} to provide upper bounds on $\sigma(n,k)$ for $k \ge 3$. 

Our goal in this section is to find a lower bound on $c(n,m)$. We begin by showing that families that minimize comparability are `convex'.

\begin{lem}\label{ClosedUnderContainment}
Let $\F \subseteq \mathcal{P}([n])$. Let $\F' := \F \cup \{ Z \in \mathcal{P}([n]) : X \subseteq Z \subseteq Y, \text{ where } X,Y \in \F\}$. Then $c(\F') = c(\F)$.
\end{lem}

\begin{proof}[Proof of Lemma \ref{ClosedUnderContainment}]
Let $Z$ be a set such that $X \subseteq Z \subseteq Y$, for some $X,Y \in \mathcal{F}$. Observe that any set in $\mathcal{P}([n])$ that is comparable to $Z$ is either comparable to $X$ or to $Y$. So $c(\F \cup Z) = c(\F)$. Repeatedly applying this observation gives the result. 
\end{proof}

Theorem \ref{thm:compnum} can now be deduced from the Harris-Kleitman inequality. Recall that a family $\mathcal{U} \subseteq \mathcal{P}([n])$ is an \emph{upset} if for all $X \in \mathcal{U}$, if $X \subseteq Y$, then $Y \in \mathcal{U}$. A family $\mathcal{D} \subseteq \mathcal{P}([n])$ is a \emph{downset} if for all $X \in \mathcal{D}$, if $Y \subseteq X$, then $Y \in \mathcal{D}$. 

\begin{lem}[Harris-Kleitman Inequality \cite{kleitman}] \label{harriskleitman}
Let $ \mathcal{U} \subseteq \mathcal{P}([n])$ be an upset and $\mathcal{D} \subseteq \mathcal{P}([n])$ be a downset. Then 
\[
    \frac{|\mathcal{U} \cap \mathcal{D}|}{2^{n}} \leq \frac{|\mathcal{U|}}{2^{n}} \cdot \frac{|\mathcal{D}|}{2^{n}}.
\]
\end{lem}
We will apply Lemma~\ref{harriskleitman} to prove a lower bound on $c(n,m)$. For convenience, for a family $\F \subseteq \mathcal{P}([n])$, define 
\[
    \mathcal{U}_{\F} = \{X \in \mathcal{P}([n]) : F \subseteq X \text{ for some } F \in \F\}
\]
and 
\[
    \mathcal{D}_{\F} = \{X \in \mathcal{P}([n]) : X \subseteq F \text{ for some } F \in \F\}.
\]
\begin{theorem} \label{thm:compnum} For $1 \leq m \leq 2^n$,
$$c(n,m) \geq 2^{n/2 + 1}\sqrt{m} - m.$$
\end{theorem}

\begin{proof}
Let $\F \subseteq \mathcal{P}([n])$ be such that $|\F| = m$ and $c(\F) = c(n,m)$. We may assume $\F$ is convex. If not, by Lemma~\ref{ClosedUnderContainment} we may add sets to make it convex and then remove minimal or maximal elements to obtain $\F'$ such that $|\F'| = |\F|$ and $c(\F') \le c(\F)$.
Note that $c(\F) = |\mathcal{U}_{\F}| + |\mathcal{D}_{\F}| - |\mathcal{U}_{\F}\cap\mathcal{D}_{\F}|$. Since $\F$ is convex, $|\F| = |\mathcal{U}_{\F}\cap\mathcal{D}_{\F}| = m$. Using the AM-GM inequality we get 
\[
    c(\F) \geq 2\sqrt{|\mathcal{U}_{\F}||\mathcal{D}_{\F}|} - m.
\]
Since $\mathcal{U}_{\F}$ is an upset and $\mathcal{D}_{\F}$ is a downset, we apply Lemma \ref{harriskleitman} to get 
\[
    c(\F) \geq 2\sqrt{2^{n}m} - m = 2^{\frac{n}{2}+1}\sqrt{m}-m,
\]
as required.
\end{proof}

It is now a simple consequence of Theorem~\ref{thm:compnum} to see that \eqref{eq:sumpair} holds for all $n$. 

\begin{theorem}\label{thm:sumpair}
Let  $(\F, \G)$ be cross-Sperner in $\mathcal{P}([n])$. Then 
\[
    |\F| + |\G| \leq 2^n - 2^{\floor{n/2}} - 2^{\ceil{n/2}} + 2.
\]
\end{theorem}
\begin{proof}
Let $(\F,\G) \in \mathcal{P}([n])^2$ be a cross-Sperner pair. Suppose $|\F| = m$. Since $\F, \G \neq \emptyset$, $1 \leq m \leq 2^{n} - 1$. 
Moreover, we may assume without loss of generality that $|\F| \le |\G|$. We know $|\F||\G| \le 2^{2n-4}$ by (\ref{eq:prod}), which implies that $m \le 2^{n-2}$.

Then, $c(\F) \ge |\mathcal{U}_{\F}| + |\mathcal{D}_{\F}| - m$ and $|\G| \leq 2^{n} - c(\F) \le |\mathcal{U}_{\F}| - |\mathcal{D}_{\F}| + m$. 
Thus \begin{equation} \label{eq:sum}
    |\F| + |\G| \leq 2^{n} - |\mathcal{U}_{\F}| - |\mathcal{D}_{\F}| + 2m.
\end{equation}
We have the following two cases. 
\begin{case} Suppose $m = 1$. Since $\F$ only consists of one set, say $F$, we have $\mathcal{U}_{\F} = 2^{n - |F|}$ and $\mathcal{D}_{\F} = 2^{|F|}$. So $|\mathcal{U}_{\F}||\mathcal{D}_{\F}| = 2^{n}$, which, by AM-GM gives $|\mathcal{U}_{\F}| + |\mathcal{D}_{\F}| \geq 2^{n/2 + 1} \ge 2^{\floor{n/2}} + 2^{\ceil{n/2}}$. So \eqref{eq:sum} yields
\[
    |\F| + |\G| \leq 2^{n} - 2^{\floor{n/2}} - 2^{\ceil{n/2}} + 2,
\]
as required. This completes the case $m=1$.
\end{case}

\begin{case} Now suppose $m \geq 2$. By Theorem \ref{thm:compnum}, $|\mathcal{U}_{\F}| + |\mathcal{D}_{\F}| \geq 2^{\frac{n}{2} + 1} \sqrt{m}$, so Equation \eqref{eq:sum} gives
\[
    |\F| + |\G| \leq 2^{n} - 2^{\frac{n}{2}+1}\sqrt{m} + 2m.
\]
By differentiation with respect to $m$ we see that the expression on the right-hand side is decreasing in the range $2 \le m \leq 2^{n-2}$. It is therefore maximized at $m = 2$, where we have 
\[
    2^n - 2^{\frac{n}{2}+1}\sqrt{m} + 2m = 2^n - 2^{\frac{n+3}{2}} + 4.
\]
Note that for all $n \geq 2$,
\[
    2^{\frac{n+3}{2}} - 4 \geq 2^{\floor{n/2}} + 2^{\ceil{n/2}} - 2.
\]
This implies that  $2^{n} - 2^{\floor{n/2}} - 2^{\ceil{n/2}} + 2 \geq  2^{n} - 2^{\frac{n}{2}+1}\sqrt{m} + 2m$ for all $2 \leq m \leq 2^{n-2}$ and $n \geq 2$. This completes the case $m \ge 2$.
\end{case}
\setcounter{case}{0}

We conclude that $|\F| + |\G| \leq 2^{n} - 2^{\floor{n/2}} - 2^{\ceil{n/2}} + 2$, as desired. 
\end{proof}

\section{Bounding $\maxprod(n,k)$}\label{sec:prod}

The goal of this section is to prove Theorems~\ref{thm:prodl} and \ref{thm:produ}.
\subsection{Lower Bound on $\maxprod(n,k)$}

Theorem \ref{thm:prodl} follows directly from the following (slightly stronger) statement.
\begin{lem}\label{lem:prodlower}
    Let $n,k$ be integers with $k \ge 2$ and $n > k\log_2{k} + k$. Then 
    \[
    \maxprod(n,k) \ge
    \left(
        \left(\frac{1}{k} -\frac{1}{2^{\floor{n/k}}}\right)
        \left(1 - \frac{1}{k}\right)^{k-1}
    \right)^k 2^{kn}
    \]
\end{lem}
\begin{proof}
    Partition $[n]$ into $k$ parts $A_1, A_2, \cdots, A_k$ each of size $\floor{\frac{n}{k}}$ or $\ceil{\frac{n}{k}}$. For each $1 \leq i \leq k$, take $\X_i$ to be an initial segment of colex in $\mathcal{P}(A_i)$ such that $|\X_i| = \lambda_i 2^{|A_i|}$ for some $0 < \lambda_i < 1$. Set $\Y_i := \mathcal{P}(A_i) \setminus \X_i$.
    Now we construct a cross-Sperner system $(\F_1, \F_2, \cdots, \F_k)$. Define \begin{align*}
        \F_i := \{ F \in \mathcal{P}([n]) : F \cap A_i \in \X_i, F \cap A_j \in \Y_j \text{ for all } j \neq i\}.
    \end{align*}
    Refer to Example~\ref{ex:const} for an example of this construction.
    
    To see that $(\F_1, \F_2, \cdots, \F_k)$ is cross-Sperner in $\mathcal{P}([n])$, consider $S \in \F_i$ and $T \in \F_j$. We must show that $S$ and $T$ are incomparable. If $S \subseteq T$. Then $S \cap A_j \subseteq T \cap A_j$, so there is some $Y \in \Y_j$ and $X \in \X_j$ such that $Y \subseteq X$, a contradiction. Analogously, we see that $T$ cannot be a subset of $S$. Hence  $(\F_1, \F_2, \cdots, \F_k)$ is cross-Sperner as required.

    Observe that
    \[
    |\F_i| = |\X_i|\prod_{j \neq i} |\Y_j|,
    \]
    and so 
    $$\pi(n,k) \ge \prod_{i = 1}^{k} |\F_i| = \prod_{i = 1}^{k} \left( |\X_i|\prod_{j \neq i} |\Y_j|\right).$$ 
    To complete the proof of Lemma~\ref{lem:prodlower} it remains to optimise the sizes of the $\lambda_i$. We have
    \begin{equation*}
        |\F_i| = \lambda_i 2^{|A_i|} \prod_{j \neq i} (1-\lambda_j)2^{|A_j|} = \lambda_i 2^{|A_1| + |A_2| + \cdots |A_k|} \prod_{j \neq i} (1 - \lambda_j) = \lambda_i 2^{n} \prod_{j \neq i} (1 - \lambda_j).
    \end{equation*}
    So \begin{equation} \label{eq:prodlower}
        \prod_{i = 1}^{k} |\F_i| = \left(\prod_{i = 1}^{k} \lambda_i(1 - \lambda_i)^{k-1}\right)2^{kn}
    \end{equation}
    For each $1 \le k \le i$, set $\lambda_i = \frac{1}{2^{|A_i|}}\floor{\frac{2^{|A_i|}}{k}}$.
    We have 
    \[
    \frac{1}{k} - \frac{1}{2^{\floor{n/k}}} \le
     \frac{1}{k} - \frac{1}{2^{|A_i|}} \le
      \lambda_i \le
    \frac{1}{k}.
    \]
    For $n > k\log_2{k} + k$ we have $2^{-\floor{n/k}} \le 2^{-(n/k - 1)} < \frac{1}{k}$ and so $\lambda_i$ is not zero.
    Therefore, with this choice of $\lambda_i$ we get
    \[
        \prod_{i = 1}^{k} |\F_i| \ge
        \left(\left(\frac{1}{k} - \frac{1}{2^{\floor{n/k}}}\right)\left(1 - \frac{1}{k}\right)^{k-1}\right)^{k}2^{kn},
    \]
    as required.
\end{proof}
\begin{remark}
Note that if $k$ is a power of $2$, in the proof of Lemma~\ref{lem:prodlower} we have  $\lambda_i = \frac{1}{k}$ for all $1 \le i \le k$. Therefore in this case we can eliminate the $-\frac{1}{2^{\floor{n/k}}}$ term.
\end{remark}

For clarity, we provide an example of the construction given in Lemma \ref{lem:prodlower}.
\begin{example}\label{ex:const}
Let $n = 6$ and $k = 3$. Partition $[6]$ into \[ A_1 = \{1,2\}, A_2 = \{3,4\}, A_3 = \{5,6\}. \]  Then let \begin{align*}
    \mathcal{X}_1 = \{\emptyset, \{1\}\}\\
    \mathcal{X}_2 = \{\emptyset, \{3\}\}\\
    \mathcal{X}_3 = \{\emptyset, \{5\}\}.
\end{align*}
So \begin{align*}
    \mathcal{Y}_1 = \{\{2\}, \{1,2\}\}\\
    \mathcal{Y}_2 = \{\{4\},\{3,4\}\}\\
    \mathcal{Y}_3 = \{\{6\},\{5,6\}\}.
\end{align*}

Then we construct our cross-Sperner system to be \begin{align*}
    \F_1 = \{\{4,6\}, \{4,5,6\}, \{3,4,6\}, \{3,4,5,6\}, \{1,4,6\}, \{1,4,5,6\}, \{1,3,4,6\}, \{1,3,4,5,6\}\}\\
    \F_2 = \{\{2,6\}, \{2,5,6\}, \{2,3,6\}, \{2,3,5,6\}, \{1,2,6\}, \{1,2,5,6\}, \{1,2,3,6\}, \{1,2,3,5,6\}\}\\
    \F_3 = \{\{2,4\}, \{2,4,5\}, \{2,3,4\}, \{2,3,4,5\}, \{1,2,4\}, \{1,2,4,5\}, \{1,2,3,4\}, \{1,2,3,4,5\}\}.
\end{align*}

\end{example}
We now deduce Theorem~\ref{thm:prodl} (restated below for convenience) from Lemma~\ref{lem:prodlower}. 
\prodl*
\begin{proof}
Take $n$ sufficiently large so that 
\[
    \frac{1}{2^{\floor{n/k}}} \le \frac{1}{k} - \frac{1}{ek}\left(1 + \frac{1}{k-1}\right)^{k-1} = \frac{1}{ek} \left( e - \left(1 + \frac{1}{k-1}\right)^{k-1}\right).
\]
This is possible as $ \left(1 + \frac{1}{k-1}\right)^{k-1}$ tends to $e$ from below.
Substituting this into Lemma \ref{lem:prodlower}, we see that 
    \[
    \maxprod(n,k) \ge
    \left(
        \frac{1}{ek}\left(1 + \frac{1}{k-1}\right)^{k-1} 
        \left(1 - \frac{1}{k}\right)^{k-1}
    \right)^k 2^{kn}
    =    \left(\frac{1}{ek}
    \right)^k 2^{kn}.
    \]
\end{proof}

\subsection{Upper Bound on $\maxprod(n,k)$}
The goal of this subsection is to prove Theorem \ref{thm:produ}, restated below for convenience.
\produ*

We will use the following observation.

\begin{lem} \label{lem:combine}
    Let $1 \le j < k$ and let $(\F_1, \F_2, \ldots, \F_k) \subseteq \mathcal{P}([n])^{k}$ be cross-Sperner. Then $\left(\bigcup_{i=1}^{j}\F_i, \bigcup_{i=j+1}^{k}\F_i\right)$ is cross-Sperner in $\mathcal{P}([n])$.
\end{lem}

\begin{proof}
    Suppose for contradiction that $\left(\bigcup_{i=1}^{j}\F_i, \bigcup_{i=j+1}^{k}\F_i\right)$ is not cross-Sperner. Then there exists some $X \in \bigcup_{i=1}^{j}\F_i$ and $Y \in \bigcup_{i=j+1}^{k}\F_i$ such that $X \subseteq Y$ or $Y \subseteq X$. Since $X \in \F_i$ for some $1 \le i \le j$, and $Y \in \F_t$ for some $j+1 \le t \le k$ we deduce that $(\F_1, \F_2, \ldots, \F_k)$ is not cross-Sperner, a contradiction. 
\end{proof}

We now use Lemma~\ref{lem:combine}, along with Theorem~\ref{thm:compnum}, to give an upper bound on $\maxprod(n,k)$.

\begin{proof}[Proof of Theorem~\ref{thm:produ}]
    Suppose $(\F_1, \F_2, \cdots, \F_k)$ is cross-Sperner in $\mathcal{P}([n])$. Let $a = \floor{k/2}, b = \ceil{k/2}$, and observe that $a + b  = k$. Let $\G = \cup_{i=1}^{a} \F_i$ and $\mathcal{H} = \cup_{a + 1}^{k} \F_i$. Notice that  $(\G,\mathcal{H}) \subseteq \mathcal{P}([n])^2$ is cross-Sperner by Lemma \ref{lem:combine}, so if $|\G| = m$, then $|\mathcal{H}| \leq 2^n -c(n,m)$. Moreover, $\prod_{i=1}^{a} |\F_i| \leq \left(\frac{m}{a}\right)^{a}$ and $\prod_{j=a+1}^{k} |\F_j| \leq \left(\frac{2^n -2^{n/2 +1}\sqrt{m} + m}{b}\right)^b$ since each product is maximized when the families are of equal sizes. Thus, \begin{equation} \label{eq:produpper}
        \prod_{i=1}^{k} |\F_i| = \prod_{i=1}^{a} |\F_i| \prod_{j=a+1}^{k}|\F_j| \leq \left(\frac{m}{a}\right)^{a}\left(\frac{2^n -2^{n/2 +1}\sqrt{m} + m}{b}\right)^b := h(m).
    \end{equation} 
    To find an upper bound on the left hand side of \eqref{eq:produpper}, we differentiate with respect to $m$ to find the value of $m$ that maximises the right hand side. 
    \begin{equation*}
        \frac{d}{dm}h(m) = \left(\frac{m}{a}\right)^a\left(\frac{(2^{n/2} - \sqrt{m})^2}{b}\right)^b(a(\sqrt{m}-2^{n/2})+b\sqrt{m})(m^{3/2}-m2^{n/2})^{-1}.
    \end{equation*}
    Setting this equal to zero yields $m \in \{0, 2^{n},\frac{a^{2}2^n}{k^2}\}$. A simple calculation shows that \eqref{eq:produpper} is maximized when $m = \frac{a^{2}2^n}{k^2}$. Thus 
    \[
    \prod_{i=1}^{k} |\F_i| \leq \left(\frac{2^n}{k^2}\right)^ka^ab^b =
    \left(\frac{2^n}{k^2}\right)^k
    \floor{\frac{k}{2}}^{\floor{k/2}}
    \ceil{\frac{k}{2}}^{\ceil{k/2}},
    \]
    as required.
\end{proof}
Note that for $k$ even, the upper bound given by Theorem \ref{thm:produ} is $\left(\frac{2^n}{2k}\right)^k.$
For $k$ odd, it is not hard to check that the upper bound is less than
$\left(1+\frac{1}{k}\right)\left(\frac{2^n}{2k}\right)^k.$

\section{Bounding $\sigma(n,k)$}\label{sec:sum}

The goal of this section is to prove Theorem~\ref{thm:cleansum}. 

\subsection{Lower Bound on $\sigma(n,k)$}
For our proof of the lower bound in Theorem~\ref{thm:cleansum} we need the following counting lemma.

\begin{lem} \label{lem:antichain}
   Let $\A := \{F_1, F_2, \cdots F_{k-1}\}$ be an antichain in $\mathcal{P}([n])$ where $F_i := \{i\} \cup \{n- \ell + 1, \ldots, n\}$. Then $c(\A) = k2^{\ell} +2^{n - \ell}\left(1 - \frac{1}{2^{k-1}}\right) - (k-1).$ 
\end{lem}

\begin{proof}
        For each $i$, let $\mathcal{S}_i$ be the collection of sets comparable to $F_i$. For ease of notation, let $G := \{n- \ell + 1, \ldots, n\}$ Observe that 
        \begin{equation}\label{eq:S1}
            |\mathcal{S}_i| = |\mathcal{U}_{F_i} \cup \mathcal{D}_{F_i}| = 2^{\ell+1} + 2^{n-\ell-1} -1, 
        \end{equation}    
        since $|F_i| = \ell+1$ and $\mathcal{U}_{F_i} \cap \mathcal{D}_{F_i} = \{F_i\}$. 

        Note that for each $i > 1$, we have 
        \begin{equation}\label{eq:down}
        \mathcal{D}_{F_i} \setminus \bigcup_{i < j} \mathcal{D}_{F_j} = \mathcal{D}_{F_i} \setminus \mathcal{D}_{F_1} = \{\{i\} \cup Y : Y \subsetneq G\}.
        \end{equation}
        Similarly, observe that for each $i > 1$, we have
        \begin{equation}\label{eq:up}
            \mathcal{U}_{F_i}\setminus \bigcup_{j < i} \mathcal{U}_{F_i} = \{Z \subseteq [n]: Z \supseteq F_i, Z \cap \{1,\ldots,i-1\} = \emptyset \}. 
        \end{equation}
        
    So now putting together \eqref{eq:S1} (to bound $|\mathcal{S}_1|$), \eqref{eq:down}, and \eqref{eq:up}, we obtain
        \begin{align*}
            \left|\bigcup_{i = 1}^{k-1} \mathcal{S}_i \right| &= |\mathcal{S}_1| + \sum_{i = 2}^{k-1} \left|\mathcal{D}_{F_i} \setminus \bigcup_{i < j} \mathcal{D}_{F_j}\right| + \sum_{i = 2}^{k-1}  \left|\mathcal{U}_{F_i} \setminus \bigcup_{i < j} \mathcal{U}_{F_j}\right| - (k-2)\\
            &= 2^{\ell+1} + 2^{n-\ell-1} -1 + (k-2)2^{\ell} + \left(\sum_{i=2}^{k-1}2^{n-\ell-i}\right) - (k-2).
        \end{align*}
        The final term occurs as the sets $F_i$ are counted both in their downset and their upset. Simplifying we get 
        \[c(\A) = k2^{\ell} +2^{n - \ell}\left(1 - \frac{1}{2^{k-1}}\right) - (k-1).\]
    \end{proof}


    

We now prove the lower bound given in Theorem \ref{thm:cleansum}. We actually prove a slightly stronger statement. 
\begin{lem} \label{lem:sumlower}
    Let $n,k \in \mathbb{N}$ where $n \ge 2k-1 - \log_2{k}\ge 1$. Then
    \[
    \s(n,k) \ge 2^{n} - \frac{3}{\sqrt{2}}\left(1 - \frac{1}{2^{k-1}}\right)^{\frac{1}{2}}\sqrt{2^n k} + 2(k-1). 
    \]
\end{lem}

\begin{proof}
    Let $a$ be an integer with the same partity as $n$ to be specified later. Let $G : = \{ n - \frac{n - a}{2} + 1, \ldots, n\}$. 
    Let $\A = \{F_1, F_2, \cdots F_{k-1}\}$ be an antichain in $\mathcal{P}([n])$, where $F_i = G \cup \{i\}$. 
This is possible as long as $n - \frac{n-a}{2} \ge k-1$, that is, $n \ge 2(k-1) - a$.
    
    By Lemma \ref{lem:antichain} (setting $\ell = \frac{n - a}{2}$), we obtain 
    \[
        c(\A) = k2^{\frac{n - a}{2}} +  2^{\frac{n + a}{2}}\left(1 - \frac{1}{2^{k-1}}\right)  - (k-1).
    \]
    Define $\F_i := \{F_i\}$ for $1 \leq i \leq k-1$ and $\F_k := \{Z \subseteq [n] : Z \text{ incomparable to } F_i \text{ for all } 1 \leq i \leq k-1\}$. By construction, $(\F_1,\ldots, \F_k)$ is cross-Sperner in $\mathcal{P}([n])$. 
    We have 
    \begin{align} \label{eq:sumlower}
        \sum_{i=1}^{k} |\F_i| &= (k-1) + 2^n - c(\A) \nonumber \\ 
    &= 2^{n} - \sqrt{2^n}\left(\frac{k}{\sqrt{2^a}} +  \left(1 - \frac{1}{2^{k-1}}\right)\sqrt{2^a}\right) + 2(k-1). 
    \end{align}
    Differentiating this expression with respect to $a$ gives 
    \[
        -\frac{\ln{2}}{2}\sqrt{2^n}\left(- \frac{k}{\sqrt{2^a}} + \left(1 - \frac{1}{2^{k-1}}\right)\sqrt{2^a} \right).
    \]
    
    Thus we can see that if there were no restrictions on $a$ the maximum value of (\ref{eq:sumlower}) would be achieved when $2^a = k\frac{2^{k-1}}{2^{k-1}-1}$; that is, $ a = {\log_2(k)+ \log_2{\left(\frac{2^{k-1}}{2^{k-1}-1}\right)}}$.
   However, we require $a$ to be an integer with the same parity as $n$. Set $a$ to be the unique such integer such that
   \[
    - 1 
    < a - {\log_2(k) - \log_2{\left(\frac{2^{k-1}}{2^{k-1}-1}\right)}} 
    \le 1
    \] 
    and let $c = a - {\log_2(k) -\log_2{\left(\frac{2^{k-1}}{2^{k-1}-1}\right)}}$. Note that $n \ge 2k-1 - \log_2{k}\ge 1$ by hypothesis. This ensures that $n \ge 2(k-1) - a$ for any such value of $a$.
    We have 
    \begin{align} \label{eq:sumlower_c}
        \sum_{i=1}^{k} |\F_i| 
        &= 2^{n}  - \sqrt{2^{n}}\left(\frac{k}{\sqrt{2^a}} +  \left(1 - \frac{1}{2^{k-1}}\right)\sqrt{2^a}\right) 
        + 2(k-1)
        \nonumber \\
        &= 2^{n} 
        - \sqrt{2^nk}\left(1 - \frac{1}{2^{k-1}}\right)^{1/2}\left(\frac{1}{\sqrt{2^c}} + \sqrt{2^c}\right) 
        + 2(k-1) \\
        &= 2^{n} - \sqrt{2^n k}\left(1 - \frac{1}{2^{k-1}}\right)^{1/2}\left(\frac{3}{\sqrt{2}}\right) 
        + 2(k-1)\nonumber
    \end{align}
    where the last inequality follows from the fact that the bracketed expression in (\ref{eq:sumlower_c}) is maximised when $c = 1$ for $c$ in the range $-1 < c \le 1$. 
\end{proof}

For certain values of $k$ we can prove a stronger lower bound which essentially matches the upper bound given in Lemma \ref{lem:sumupper}.
\begin{cor}
    Let $n,k \in \mathbb{N}$ and suppose that $k = 2^a$ where $a$ has the same parity as $n$ and $n \ge 2(k-1) - a$. Then 
    \[
        \s(n,k) \geq 2^n  - 2\sqrt{2^{n}k}\left(1 - \frac{1}{2^{k}} \right) +2(k-1).
    \]
\end{cor}
\begin{proof}
    Apply the proof of Lemma \ref{lem:sumlower} with $a = \log_2{k}$.
\end{proof}


\subsection{Upper Bound on $\sigma(n,k)$}

\begin{lem}\label{lem:sumupper}
    For $k \ge 2$ and $n$ such that $2^n \ge (k-1)(1 + \sqrt{k-1})^2$,  
    \[ \s(n,k) \leq  2^n  - 2\sqrt{2^n(k-1)} + 2(k-1).\]
\end{lem}

\begin{proof}
    Suppose $(\F_1, \F_2, \cdots, \F_k)$ is cross-Sperner in $\mathcal{P}([n])$. We may and will assume that $|\F_1| \leq |\F_2| \leq \cdots \leq |\F_k|$. Define $\mathcal{G} := \cup_{i=1}^{k-1} \F_i$. Let $m = |\mathcal{G}|$ and observe that, as each family is non-empty, we have $m \ge k-1$. 

    Note that $|\F_k| \le 2^{n} - c(n,m) \le 2^{n} - 2^{n/2 +1}\sqrt{m} +m = \left(\sqrt{2^n} - \sqrt{m}\right)^2$.
    Since the families are ordered by increasing size, $|\F_k| \geq \frac{m}{k-1}$. Putting this together gives
\[
    \frac{m}{k-1} \le |\F_k| \le \left(\sqrt{2^n} - \sqrt{m}\right)^2.
\]
    Rearranging, we obtain
\begin{equation}\label{eq:mup}
     \sqrt{m} \le \sqrt{2^n}\left(\frac{\sqrt{k-1}}{1+\sqrt{k-1}}\right).
\end{equation}

    Now consider the sum
    \begin{equation} \label{eq:sumupper}
            \sum_{i=1}^{k} |\F_i| = |\mathcal{G}| + |\F_k|
             \leq m + \left(\sqrt{2^n} - \sqrt{m}\right)^2.
    \end{equation}
    Let $x = \tfrac{1}{2}\sqrt{2^{n}} - \sqrt{m}$. Substituting this into expression (\ref{eq:sumupper}) above gives 
    \[
        \left(\tfrac{1}{2}\sqrt{2^{n}} - x\right)^2 + \left(\tfrac{1}{2}\sqrt{2^{n}} + x\right)^2 
        = 2^{n-1} + 2x^2
    \]
    and it is clear that (\ref{eq:sumupper}) is maximised when $|x| = |\tfrac{1}{2}\sqrt{2^{n}} - \sqrt{m}|$ is as large as possible. 
    Combining $m \ge k-1$ with \eqref{eq:mup} gives $\sqrt{k-1} \le \sqrt{m} \le \sqrt{2^n}\left(\frac{\sqrt{k-1}}{1+\sqrt{k-1}}\right)$, we need only find which of these end values is further from $\tfrac{1}{2}\sqrt{2^{n}}$.
 
If we have $2^n \ge (k-1)(1+\sqrt{k-1})^2$ then
\[
\tfrac{1}{2}\sqrt{2^{n}} - \sqrt{k-1} > \tfrac{1}{2}\sqrt{2^{n}} - \frac{\sqrt{2^n}}{1 + \sqrt{k-1}} = \sqrt{2^n}\left(\frac{\sqrt{k-1}}{1+\sqrt{k-1}}\right) - \tfrac{1}{2}\sqrt{2^{n}}
\]
and thus expression (\ref{eq:sumupper}) is maximised when $m = k-1$.
Substituting $m = k-1$ into (\ref{eq:sumupper}) gives
\begin{align*}
 \sum_{i=1}^{k} |\F_i| &\le  (k-1) + \left(\sqrt{2^n} - \sqrt{k-1}\right)^2 \\
 &= 2^n - 2\sqrt{2^n(k-1)} + 2(k-1).
\end{align*}
\end{proof}

\begin{proof}[Proof of Theorem~\ref{thm:cleansum}]
Lemmas \ref{lem:sumlower} and \ref{lem:sumupper} together give Theorem \ref{thm:cleansum}. Observe that $2^{2k} \ge {(k-1)(1+\sqrt{k-1})^2}$ for $k \ge 2$ so the conditions of Lemma \ref{lem:sumupper} hold.
\end{proof}
\section{Closing remarks}\label{sec:conc}
In Section~\ref{sec:prod} we provide upper and lower bounds on $\maxprod(n,k)$ in Theorems~\ref{thm:prodl} and \ref{thm:produ}. Comparing these bounds shows that they differ by a factor of $\left(\frac{e}{2}\right)^k$ for $k$ even and less than $\left(1+\frac{1}{k}\right)\left(\frac{e}{2}\right)^k$ for $k$ odd. It would be interesting to tighten this gap. We believe that (for large $n$) the bound given in Lemma~\ref{lem:prodlower} ought to be essentially best possible. 

\begin{conj} Let $k \ge 2$ be fixed and $n$ be sufficiently large with respect to $k$. Then
\[
\pi(n,k) = (1+o(1)) \left(\frac{(k-1)^{k-1}}{k^k}2^n\right)^k.
\]
\end{conj}
Our lower bound on $\pi(n,k)$ holds in the case $n > k\log_2{k} + k$. For small fixed values of $n$ and $k$, we also have some bounds for $\pi(n,k)$, see~\cite{Akina}. In particular, we have $f(4,3)=9$, $f(5,3)\ge 81$ $f(6,3) \geq 810$ and $f(5,4) \geq 108$. \\


\textbf{Note added before submission:} In the final stages of preparation of this article, we noticed a recent paper of Gowty, Horsley, and Mammoliti~\cite{Gowty}, concerning the comparability number. They give a very different proof of Theorem~\ref{thm:compnum} (see Corollary 1.2 of \cite{Gowty}) and use it as we do to deduce Theorem~\ref{thm:sumpair}. They also provide some very interesting further analysis of the comparability number and sets that minimise $c(n,m)$.

\bibliographystyle{abbrv}
\bibliography{ref}

\newpage

\end{document}